\documentclass[a4paper,11pt]{article}

\usepackage[utf8]{inputenc}
\usepackage[english,activeacute]{babel}
\usepackage{amsmath,amsthm,amsfonts,amssymb}
\usepackage{mathpartir}
\usepackage{ stmaryrd }
\usepackage{graphics}
\usepackage{enumerate}
\usepackage{mathrsfs}

\input{xy}
\usepackage[all]{xy}

\theoremstyle{plain}
\newtheorem{thm}{Theorem}[subsection]
\newtheorem{cor}[thm]{Corollary}
\newtheorem{lemma}[thm]{Lemma}
\newtheorem{proposition}[thm]{Proposition}
\newtheorem{rmk}[thm]{Remark}
\newtheorem{defs}[thm]{Definition}
\theoremstyle{remark}

\input{diagxy}
\xyoption{curve}

\newbox\anglebox 
\setbox\anglebox=\hbox{\xy \POS(75,0)\ar@{-} (0,0) \ar@{-} (75,75)\endxy}
 
\newbox\angleboxr 
\setbox\angleboxr=\hbox{\xy \POS(0,0)\ar@{-} (0,75) \ar@{-} (75,0)\endxy}
 
\newbox\sanglebox 
\setbox\sanglebox=\hbox{\xy \POS(50,0)\ar@{-} (0,0) \ar@{-} (50,50)\endxy}
 
\newbox\sangleboxr 
\setbox\sangleboxr=\hbox{\xy \POS(0,0)\ar@{-} (0,50) \ar@{-} (50,0)\endxy}
 
\newbox\sangleboxf 
\setbox\sangleboxf=\hbox{\xy \POS(50,50)\ar@{-} (50,0) \ar@{-} (0,50)\endxy}
 
\newbox\angleboxf 
\setbox\angleboxf=\hbox{\xy \POS(75,75)\ar@{-} (75,0) \ar@{-} (0,75)\endxy}
 
\newbox\sangleboxfr 
\setbox\sangleboxfr=\hbox{\xy \POS(0,50)\ar@{-} (50,50) \ar@{-} (0,0)\endxy}
 
\newbox\angleboxfr 
\setbox\angleboxfr=\hbox{\xy \POS(0,75)\ar@{-} (75,75) \ar@{-} (0,0)\endxy}

\newcommand{\Sets}{\ensuremath{\mathbf{Set}}}

\newcommand{\theory}{\ensuremath{\mathbb{T}}}

\title{INFINITARY GENERALIZATIONS OF DELIGNE'S COMPLETENESS THEOREM}
\author{Christian Esp\'indola}

\begin{document}
\date{}
\maketitle

\begin{abstract}
Given a regular cardinal $\kappa$ such that $\kappa^{<\kappa}=\kappa$ (or any regular $\kappa$ if the Generalized Continuum Hypothesis holds), we study a class of toposes with enough points, the $\kappa$-separable toposes. These are equivalent to sheaf toposes over a site with $\kappa$-small limits that has at most $\kappa$ many objects and morphisms, the (basis for the) topology being generated by at most $\kappa$ many covering families, and that satisfy a further exactness property $T$. We prove that these toposes have enough $\kappa$-points, that is, points whose inverse image preserve all $\kappa$-small limits. This generalizes the separable toposes of Makkai and Reyes, that are a particular case when $\kappa=\omega$, when property $T$ is trivially satisfied. This result is essentially a completeness theorem for a certain infinitary logic that we call $\kappa$-geometric, where conjunctions of less than $\kappa$ formulas and existential quantification on less than $\kappa$ many variables is allowed. We prove that $\kappa$-geometric theories have a $\kappa$-classifying topos having property $T$, the universal property being that models of the theory in a Grothendieck topos with property $T$ correspond to $\kappa$-geometric morphisms (geometric morphisms the inverse image of which preserves all $\kappa$-small limits) into that topos. Moreover, we prove that $\kappa$-separable toposes occur as the $\kappa$-classifying toposes of $\kappa$-geometric theories of at most $\kappa$ many axioms in canonical form, and that every such $\kappa$-classifying topos is $\kappa$-separable. Finally, we consider the case when $\kappa$ is weakly compact and study the $\kappa$-classifying topos of a $\kappa$-coherent theory (with at most $\kappa$ many axioms), that is, a theory where only disjunction of less than $\kappa$ formulas are allowed, obtaining a version of Deligne's theorem for $\kappa$-coherent toposes from which we can derive, among other things, Karp's completeness theorem for infinitary classical logic.
\end{abstract}

\noindent $\mathbf{Keywords:}$ classifying topos, infinitary logics, completeness theorems, sheaf models.

\section{Introduction}

This paper is a continuation of the investigation begun in \cite{espindola} on infinitary categorical logic, focusing now on infinitary generalizations of Deligne's completeness theorem. This theorem asserts that a coherent topos has enough points, which is essentially G\"odel completeness theorem for finitary first-order classical logic. Makkai and Reyes in \cite{mr} prove that the same is true for the so called separable toposes, those toposes of sheaves on a site that has countably many objects and morphisms and whose topology is generated by countably many covering families. This result is in turn related to the completeness of countably axiomatized theories in $\mathcal{L}_{\omega_1, \omega}$. It turns out, as we prove in the present paper, that this result can be generalized in a way that $\omega$ is replaced with any regular cardinal $\kappa$ such that $\kappa^{<\kappa}=\kappa$ (a condition satisfied for inaccessible $\kappa$, or, under the Generalized Continuum Hypothesis, for any regular $\kappa$), and that, naturally, the result is essentially a completeness theorem for what we call $\kappa$-geometric logic. This logic is an extension of geometric logic in which arities of function and relation symbols are cardinals less than $\kappa$, and one can take conjunctions of less than $\kappa$ many formulas and existential quantification of less than $\kappa$ many variables. The theory of well-orderings, for instance (see \cite{dickmann}), cannot be expressed in finite-quantifier languages, but it is $\kappa$-geometric.

In this more expressive extension there are valid sequents that cannot be derived from the usual axioms of geometric logic even if one extends the usual rules and axioms for conjunction and existential quantification. For example, the axiom of choice is expressible in the form:

$$\bigwedge_{i<\gamma}\exists \mathbf{x_i}\phi(\mathbf{x_i}) \vdash_{\mathbf{x}} \exists_{i<\gamma} \mathbf{x_i} \bigwedge_{i<\gamma} \phi(\mathbf{x_i})$$
\\
and is certainly valid (in all $\mathcal{S}et$-valued models) but not derivable. 

However, this more expressive extension valid sequents can be made derivable in $\kappa$-geometric logic with the addition of one special rule, which is a refined version of transfinite the rule of transitivity introduced in \cite{espindola}. This rule corresponds to an exactness condition $T$ satisfied by $\kappa$-geometric categories (the categories associated to $\kappa$-geometric theories), and it turns out that $\kappa$-geometric theories have what we call a $\kappa$-classifying topos. This is a Grothendieck topos satisfying the exactness condition $T$, where there is a generic model of the theory whose image along the inverse image of $\kappa$-geometric morphisms (those geometric morphisms whose inverse image preserve $\kappa$-small limits) corresponds precisely to the models of the theory in any other topos that satisfies the same exactness condition $T$. 

In the same way that countably axiomatized geometric theories are complete (with respect to $\mathcal{S}et$-valued models), one can prove that $\kappa$-geometric theories with at most $\kappa$ many axioms are also complete. The restriction on the cardinality of the axiomatization will then imply that the corresponding $\kappa$-classifying topos will be $\kappa$-separable, and by completeness it will be possible to prove that it has enough $\kappa$-points (points whose inverse image preserve $\kappa$-small limits). Since property $T$ is valid in $\mathcal{S}et$, the existence of enough $\kappa$-points for a given topos necessarily implies that the topos has property $T$, whence its naturalness in the definition of $\kappa$-separable toposes. Indeed, a localic topos $\mathcal{S}h(L)$ without points cannot be $\kappa$-separable for $\kappa>2^L$, even if its site does have $\kappa$-small limits, at most $\kappa$ many objects and morphisms, and the (basis for the) topology is generated by at most $\kappa$ many covering families.

In the particular case when $\kappa$ is a weakly compact cardinal, the exactness property $T$ on the $\kappa$-classifying topos of a $\kappa$-coherent theory (a $\kappa$-geometric theory where all disjunctions in the axioms are indexed by ordinals less than $\kappa$) that is axiomatized with at most $\kappa$-many axioms, adopts the form of the transfinite transitivity rule of \cite{espindola}. In this case, the topos is equivalent to sheaves on a site where Grothendieck topology is generated by families of less than $\kappa$ many morphisms, and its $\kappa$-separability, that implies that it has enough $\kappa$-points, is the precise generalization of Deligne's completeness theorem from coherent toposes to $\kappa$-coherent toposes. 

It turns out that the corresponding completeness theorem for $\kappa$-coherent theories (associated with Deligne's theorem for $\kappa$-coherent toposes) adopts, via Morleyization, the form of Karp's completeness theorem for $\mathcal{L}_{\kappa, \kappa}$, the exactness property $T$ corresponding to a combination of her distributivity and dependent choice axioms. 

We will also see that in the particular case $\kappa=\omega$, our result is just Makkai and Reyes result for separable toposes, since the exactness property $T$ in a separable topos is always satisfied.

\subsection{$\kappa$-geometric logic}

Let $\kappa$ be a regular cardinal. The syntax of $\kappa$-geometric logic consists of a (well-ordered) set of sorts and a set of function and relation symbols, these latter together with the corresponding type, which is a subset with less than $\kappa$ many sorts. Therefore, we assume that our signature may contain relation and function symbols on $\gamma<\kappa$ many variables, and we suppose there is a supply of $\kappa$ many fresh variables of each sort. Terms and atomic formulas are defined as usual, and general formulas are defined inductively according to the following:

\begin{defs} If $\phi, \psi, \{\phi_{\alpha}: \alpha<\gamma\}$ (for each $\gamma<\kappa$) and $\{\psi_{\alpha}: \alpha<\delta\}$ (for each $\delta$) are $\kappa$-geometric formulas, the following are also formulas: $\bigwedge_{\alpha<\gamma}\phi_{\alpha}$, $\exists_{\alpha<\gamma} x_{\alpha} \phi$ (also written $\exists \mathbf{x}_{\gamma} \phi$ if $\mathbf{x}_{\gamma}=\{x_{\alpha}: \alpha<\gamma\}$) and $\bigvee_{\alpha<\delta}\psi_{\alpha}$, this latter provided that $\cup_{\alpha<\delta}FV(\psi_{\alpha})$, the set of free variables of all $\psi_{\alpha}$, has cardinality less than $\kappa$.
\end{defs}
 
We use sequent style calculus to formulate the axioms of first-order logic, as can be found, e.g., in \cite{johnstone}, D1.3. The system for $\kappa$-geometric logic is described in the following:

\begin{defs}\label{sfol}
 The system of axioms and rules for $\kappa$-geometric logic consists of

\begin{enumerate}
 \item Structural rules:
 \begin{enumerate}
 \item Identity axiom:
\begin{mathpar}
\phi \vdash_{\mathbf{x}} \phi 
\end{mathpar}
\item Substitution rule:
\begin{mathpar}
\inferrule{\phi \vdash_{\mathbf{x}} \psi}{\phi[\mathbf{s}/\mathbf{x}] \vdash_{\mathbf{y}} \psi[\mathbf{s}/\mathbf{x}]} 
\end{mathpar}

where $\mathbf{y}$ is a string of variables including all variables occurring in the string of terms $\mathbf{s}$.
\item Cut rule:
\begin{mathpar}
\inferrule{\phi \vdash_{\mathbf{x}} \psi \\ \psi \vdash_{\mathbf{x}} \theta}{\phi \vdash_{\mathbf{x}} \theta} 
\end{mathpar}
\end{enumerate}

\item Equality axioms:

\begin{enumerate}
\item 

\begin{mathpar}
\top \vdash_{x} x=x 
\end{mathpar}

\item 

\begin{mathpar}
(\mathbf{x}=\mathbf{y}) \wedge \phi \vdash_{\mathbf{z}} \phi[\mathbf{y}/\mathbf{x}]
\end{mathpar}

where $\mathbf{x}$, $\mathbf{y}$ are contexts of the same length and type and $\mathbf{z}$ is any context containing $\mathbf{x}$, $\mathbf{y}$ and the free variables of $\phi$.
\end{enumerate}

\item Conjunction axioms and rules:

$$\bigwedge_{i<\gamma} \phi_i \vdash_{\mathbf{x}} \phi_j$$

\begin{mathpar}
\inferrule{\{\phi \vdash_{\mathbf{x}} \psi_i\}_{i<\gamma}}{\phi \vdash_{\mathbf{x}} \bigwedge_{i<\gamma} \psi_i}
\end{mathpar}

for each cardinal $\gamma<\kappa$.

\item Disjunction axioms and rules:

$$\phi_j \vdash_{\mathbf{x}} \bigvee_{i<\gamma} \phi_i$$

\begin{mathpar}
\inferrule{\{\phi_i \vdash_{\mathbf{x}} \theta\}_{i<\gamma}}{\bigvee_{i<\gamma} \phi_i \vdash_{\mathbf{x}} \theta}
\end{mathpar}

for each cardinal $\gamma$.

\item Existential rule:
\begin{mathpar}
\mprset{fraction={===}}
\inferrule{\phi \vdash_{\mathbf{x} \mathbf{y}} \psi}{\exists \mathbf{y}\phi \vdash_{\mathbf{x}} \psi}
\end{mathpar}

where no variable in $\mathbf{y}$ is free in $\psi$.

\item Small distributivity axiom

$$\phi \wedge \bigvee_{i<\gamma} \psi_{i} \vdash_{\mathbf{x}} \bigvee_{i<\gamma}\phi \wedge \psi_{i}$$

for each cardinal $\gamma$.

\item Frobenius axiom:

$$\phi \wedge \exists \mathbf{y} \psi \vdash_{\mathbf{x}} \exists \mathbf{y} (\phi \wedge \psi)$$

\noindent where no variable in $\mathbf{y}$ is in the context $\mathbf{x}$.

\item Rule $T$:

\begin{mathpar}
\inferrule{\phi_{f} \vdash_{\mathbf{y}_{f}} \bigvee_{g \in \gamma^{\beta+1}, g|_{\beta}=f} \exists \mathbf{x}_{g} \phi_{g} \\ \beta<\kappa, f \in \gamma^{\beta} \\\\ \phi_{f} \dashv \vdash_{\mathbf{y}_{f}} \bigwedge_{\alpha<\beta}\phi_{f|_{\alpha}} \\ \beta < \kappa, \text{ limit }\beta, f \in \gamma^{\beta}}{\phi_{\emptyset} \vdash_{\mathbf{y}_{\emptyset}} \bigvee_{f \in B}  \exists_{\beta<\delta_f}\mathbf{x}_{f|_{\beta +1}} \bigwedge_{\beta<\delta_f}\phi_{f|_{\beta+1}}}
\end{mathpar}
\\
for each cardinal $\gamma$, where $\mathbf{y}_{f}$ is the canonical context of $\phi_{f}$, provided that, for every $f \in \gamma^{\beta+1}$,  $FV(\phi_{f}) = FV(\phi_{f|_{\beta}}) \cup \mathbf{x}_{f}$ and $\mathbf{x}_{f|_{\beta +1}} \cap FV(\phi_{f|_{\beta}})= \emptyset$ for any $\beta<\gamma$, as well as $FV(\phi_{f}) = \bigcup_{\alpha<\beta} FV(\phi_{f|_{\alpha}})$ for limit $\beta$. Here $B \subseteq \gamma^{< \kappa}$ consists of the minimal elements of a given bar\footnote{A bar over the tree $\gamma^{< \kappa}$ is an upward closed subset of nodes intersecting every branch of the tree.} over the tree $\gamma^{< \kappa}$, and the $\delta_f$ are the levels of the corresponding $f \in B$. 

\end{enumerate}
\end{defs}

The rule $T$ can be understood as follows. Consider $\gamma^{< \kappa}$, the $\gamma$-branching tree of height $\kappa$, i.e., the poset of functions $f: \beta \to \gamma$ for $\beta \leq \gamma$ with the order given by inclusion. Suppose there is an assignment of formulas $\phi_f$ to each node $f$ of $\gamma^{< \kappa}$. Then the rule expresses that if the assignment is done in a way that the formula assigned to each node entails the join of the formulas assigned to its immediate successors, and if the formula assigned to a node in a limit level is equivalent to the meet of the formulas assigned to its predecessors, then the formula assigned to the root entails the join of the formulas assigned to the nodes ranging among the minimal elements of a given bar over the tree $\gamma^{< \kappa}$.

\section{$\kappa$-geometric categories}

\subsection{The exactness property $T$}

The $\kappa$-geometric fragment of first-order logic, which is an extension of the usual geometric fragment, has a corresponding category which we are now going to define. Following \cite{makkai}, consider a $\kappa$-chain in a category $\mathcal{C}$ with $\kappa$-limits, i.e., a diagram $\Gamma: \gamma^{op} \to \mathcal{C}$ specified by morphisms $(h_{\beta, \alpha}: C_{\beta} \to C_{\alpha})_{\alpha \leq \beta<\gamma}$ such that the restriction $\Gamma|_{\beta}$ is a limit diagram for every limit ordinal $\beta$. We say that the morphisms $h_{\beta, \alpha}$ compose transfinitely, and take the limit projection $f_{\beta, 0}$ to be the transfinite composite of $h_{\alpha+1, \alpha}$ for $\alpha<\beta$.

Given any cardinal $\gamma$, consider the tree $S=\gamma^{<\kappa}$. We will consider diagrams $F: S^{op} \to \mathcal{C}$, which determine, for each node $f$, a family of arrows in $\mathcal{C}$, $\{h_{g, f}: C_{g} \to C_{f} | f \in \gamma^{\beta}, g \in \gamma^{\beta+1}, g|_{\beta}=f\}$. A $\gamma$-family of morphisms with the same codomain is said to be \emph{jointly covering} if the union of the images of the morphisms is the whole codomain. We say that a diagram $F: S^{op} \to \mathcal{C}$ is \emph{proper} if the $\{h_{g, f}: f \in S\}$ are jointly covering and, for limit $\beta$, $h_{f, \emptyset}$ is the transfinite composition of the $h_{f|_{\alpha+1}, f|_{\alpha}}$ for $\alpha+1<\beta$. Given a proper diagram and a bar over $S$ whose minimal node intersecting the branch $b$ has level $\delta_b$, we say that the families $\{h_{g, f}: f \in S\}$ compose transfinitely, and refer to the arrows $\{h_{g_b, \emptyset} | g \in \gamma^{\delta_b}, b \in \gamma^{<\kappa}\}$ as the transfinite composites (up to $\kappa$) of these families with respect to the bar. If in a proper diagram the transfinite composites of the $\kappa$-families of morphisms form itself a jointly covering family, we will say that the diagram is completely proper.

\begin{defs}
 A $\kappa$-geometric category is a $\kappa$-complete geometric category with complete subobject lattices where arbitrary unions are stable under pullback, and where every proper diagram is completely proper, i.e., the transfinite composites (up to $\kappa$) of jointly covering $\kappa$-families of morphisms form a jointly covering family.
\end{defs}

\begin{defs}
 A category with $\kappa$-small limits and arbitrary unions is said to have the exactness property $T$ if every proper diagram (corresponding to any given tree with any given bar over it) is completely proper.
\end{defs}

 This is evidently valid in $\mathcal{S}et$, and in fact in every presheaf category.

$\kappa$-geometric categories have an internal logic, in a signature containing one sort for each object, no relation symbols and one unary function symbol for each arrow, and axiomatized by the following sequents:

$$\top \vdash_x \it Id_X(x)=x$$
\\
for all objects $X$ (here $x$ is a variable of sort $X$);

$$\top \vdash_x f(x)=h(g(x))$$
\\
for all triples of arrows such that $f=h \circ g$ (here $x$ is a variable whose sort is the domain of $f$);

$$\top \vdash_y \exists x f(x)=y$$
\\
for all covers $f$ (here $x$ is a variable whose sort is the domain of $f$);

$$\top \vdash_x \bigvee_{i<\gamma}\exists y_i m_i(y_i)=x$$
\\
whenever the sort $A$ of $x$ is the union of $\gamma$ subobjects $m_i: A_i \rightarrowtail A$ (here $y_i$ is a variable of sort $A_i$);

$$\bigwedge_{i:I \to J}\overline{i}(x_I)=x_J \vdash_{\{x_I: I \in \mathbf{I}\}} \exists x \bigwedge_{I \in \mathbf{I}} \pi_I(x)=x_I$$

$$\bigwedge_{I \in \mathbf{I}}\pi_I(x)=\pi_I(y) \vdash_{x, y} x=y$$
\\
whenever there is a $\kappa$-small diagram $\Phi: \mathbf{I} \to \mathcal{C}$, $(\{C_I\}_{I \in \mathbf{I}}, \{\overline{i}: C_I \to C_J)\}_{i: I \to J})$ and a limit cone $\pi: \Delta_C \Rightarrow \Phi$, $(\pi_I: C \to C_I)_{I \in \mathbf{I}}$. Here $x_I$ is a variable of type $C_I$, and $x, y$ are variables of type $C$.

Functors preserving this logic, i.e., $\kappa$-geometric functors, are just geometric functors which preserve $\kappa$-limits, and they can be easily seen to correspond to structures of the internal theory in a given $\kappa$-geometric category, where we use a straightforward generalization of categorical semantics, as explained e.g. in \cite{johnstone}, D1.2.

\begin{lemma}\label{soundness}
$\kappa$-geometric logic is sound with respect to models in $\kappa$-geometric categories.
\end{lemma}

\begin{proof}
 Straightforward. The proof of soundness of proerty $T$ is similar to the proof of soundness of the transfinite transitivity rule from \cite{espindola}.
\end{proof}

\subsection{Completeness of $\kappa$-geometric logic}

The following can be considered as a completeness theorem in terms of models in $\kappa$-geometric categories:

\begin{proposition}\label{catcomp}
 If $\theory$ is a $\kappa$-geometric theory, then its syntactic category $\mathcal{C}_{\theory}$ is a $\kappa$-geometric category.
\end{proposition}

\begin{proof} Straightforward verification similar to the corresponding statement for $\kappa$-coherent categories from \cite{espindola}.
\end{proof}

The syntactic categories for $\kappa$-geometric logic can be equipped with appropriate topologies in such a way that the corresponding sheaf toposes are conservative models of the corresponding theories. Given a $\kappa$-geometric category we can define the $\kappa$-geometric coverage, where the covering families are given by families of arrows $f_i: A_i \to A$ such that the union of their images is the whole of $A$ (in particular, the initial object $0$ is covered by the empty family). We can also find (see \cite{bj}) a conservative sheaf model given by Yoneda embedding into the sheaf topos obtained with the $\kappa$-coherent coverage. As proven in \cite{bj}, the embedding preserves  arbitrary unions and $\kappa$-limits. Moreover, we have the following completeness theorem in terms of models in $\kappa$-geometric Grothendieck toposes:

\begin{lemma}\label{shemb}
 Given a $\kappa$-geometric category $\mathcal{C}$ with the $\kappa$-geometric coverage $\tau$, Yoneda embedding $y: \mathcal{C} \to \mathcal{S}h(\mathcal{C}, \tau)$ is a conservative $\kappa$-geometric functor and $\mathcal{S}h(\mathcal{C}, \tau)$ is a $\kappa$-geometric category.
\end{lemma}

\begin{proof}
 The proof that property $T$ holds in $\mathcal{S}h(\mathcal{C}, \tau)$ is similar to the proof that transfinite transitivity holds in sheaf models of $\kappa$-coherent categories, as in \cite{espindola}.
\end{proof}

\begin{rmk}
 Lemma \ref{shemb} allows to give many examples of $\kappa$-geometric toposes. Starting with a category of size at most $\kappa$ with $\kappa$-small limits (for example, the syntactic category of any cartesian theory in $\mathcal{L}_{\kappa, \kappa}$ of cardinality at most $\kappa$), we can arbitrarily choose $\kappa$-many covering families and generate a topology such that the topos has property $T$: simply note that the condition that transfinite composites (up to $\kappa$) of covering families are again covering is a closure condition on our set of initial covering families, and therefore these transfinite composites can be added to the topology in at most $\kappa^+$ iterations. Conversely, if a Grothendieck topology is such that the topos has property $T$, it follows that the topology has a basis satisfying that transfinite composites of covering families in the basis belong to the basis. This method hence yields all possible $\kappa$-geometric toposes.
\end{rmk}

\begin{defs}
 A $\kappa$-Grothendieck topology is a topology generated by a basis with the property that transfinite composites (up to $\kappa$) of basic covering families are also basic covering families.
\end{defs}

\begin{rmk}
 It is easy to prove that any Grothendieck topology is an $\omega$-topology (see the proof in \cite{espindola} that the rule $TT_\omega$ is provable from the rest of the axioms). Therefore, a separable topos in the sense of Makkai and Reyes (see \cite{mr}) is an $\omega$-geometric topos.
\end{rmk}

To prove completeness with respect to $\mathcal{S}et$-valued models we need the notion of a transfinite Beth model, adapted from \cite{espindola} to our case:

\begin{defs}\label{bethmodelt}
 A Beth model for pure $\kappa$-geometric logic over $\Sigma$ is a quadruple $\mathcal{B}=(K, \leq, D, \Vdash)$, where $(K, \leq)$ is a tree of height $\kappa$ and with a set $B$ of branches (i.e., maximal chains in the partial order) each of size $\kappa$; $D$ is a set-valued functor on $K$, and the forcing relation $\Vdash$ is a binary relation between elements of $K$ and sentences of the language with constants from $\bigcup_{k \in K}D(k)$, defined recursively for formulas $\phi$ as follows. There is an interpretation of function and relation symbols in each $D(k)$; if $R_k \subseteq D(k)^{\lambda}$ is the interpretation in $D(k)$ of the $\lambda$-ary relation symbol $R$ in the language, we have $k \leq l \implies R_k(D_{kl}(\mathbf{c})) \subseteq R_l(\mathbf{c})$ for $\mathbf{c} \subseteq D_k$, and: 
 
 \begin{enumerate}
  \item $k \Vdash R(\mathbf{s}(\mathbf{d})) \iff \forall b \in B_k \exists l \in b \qquad (R_l(\mathbf{s}(D_{kl}(\mathbf{d}))))$
  \item $k \Vdash \bigwedge_{i<\gamma}\phi_i(\mathbf{d}) \iff k \Vdash \phi_i(\mathbf{d}) \text{ for every } i<\gamma$
  \item $k \Vdash \bigvee_{i<\gamma}\phi_i(\mathbf{d}) \iff \forall b \in B_k \exists l \in b \qquad (l \Vdash \phi_i(D_{kl}(\mathbf{d})) \text{ for some } i<\gamma)$
  \item $k \Vdash \exists \mathbf{x} \phi(\mathbf{x}, \mathbf{d}) \iff \forall b \in B_k \exists l \in b \qquad \exists \mathbf{e} \subseteq D(l) (l \Vdash \phi(\mathbf{e}, D_{kl}(\mathbf{d}))$
 \end{enumerate}

A Beth model for a $\kappa$-geometric theory \theory\ is a Beth model for $\kappa$-geometric logic forcing all the axioms of the theory and not forcing $\bot$. 
\end{defs}

We have now:

\begin{proposition}  
$\kappa$-first-order logic is sound for Beth models.
\end{proposition}

\begin{proof}
The key part of the proof is to note that the following property holds: for any $\kappa$-geometric formula $\phi(\mathbf{x})$ and any node $k$ in the Beth model, we have $k \Vdash \phi(\mathbf{c}) \iff \forall b \in B_k \exists l \in b \qquad (l \Vdash \phi(D_{kl}(\mathbf{c})))$. This in turn can be easily proved by induction on the complexity of $\phi$. Using now this property, it is easy to check the validity of all axioms and rules of $\kappa$-geometric logic.
\end{proof}

We will need the following technical lemma, which corresponds to the canonical well-ordering of $\kappa \times \kappa$ from \cite{jechst}:

\begin{lemma}\label{dwo}
 For every cardinal $\kappa$ there is a well-ordering $f: \kappa \times \kappa \to \kappa$ with the property that $f(\beta, \gamma) \geq \gamma$.
\end{lemma}

\begin{proof}
 We define $f$ by induction on $\max (\beta, \gamma)$ as follows:
 
 $$f(\beta, \gamma)=\begin{cases} \sup\{f(\beta', \gamma')+1: \beta', \gamma'<\gamma\}+\beta  & \mbox{if }  \beta<\gamma  \\ \sup\{f(\beta', \gamma')+1: \beta', \gamma'<\beta\}+\beta+\gamma & \mbox{if } \gamma \leq \beta \end{cases}$$
 
\noindent which satisfies the required property (see \cite{jechst}, Theorem 3.5).
\end{proof}

We have now:

\begin{thm}\label{shcomp}
 Let $\kappa$ be a regular cardinal such that $\kappa^{<\kappa}=\kappa$. Then any $\kappa$-geometric theory of cardinality at most $\kappa$ has a Beth model.
\end{thm}

\begin{proof}
 Consider the syntactic category $\mathcal{C}_{\theory}$ of the theory and its conservative embedding in the topos of sheaves with the $\kappa$-geometric coverage, $\mathcal{C}_{\theory} \to \mathcal{S}h(\mathcal{C}_{\theory}, \tau)$. By assumption, the cardinality of the set $S$ of antecedents and consequents of axioms of the theory is at most $\kappa$. We will construct a Beth model of height $\kappa$ by transfinite recursion, where the underlying sets of the nodes will be specified as follows. We will build a contravariant functor $F$ from the underlying tree of height $\kappa$ to the syntactic category, defined recursively on the levels of the tree; the underlying domain corresponding to a node $q$ will be a subset of the set of arrows from $F(q)$ to the object $[x, \top]$ in the syntactic category, and the function between the underlying set of a node $q$ and that of its successor $p$ for $f: q\to p$ is given by composition with the arrow $F(f)$.

Let the image by $F$ of the root of the underlying tree be assigned the terminal object $1$, and choose as the underlying domain the set of all constants symbols $c: 1 \to [x, \top]$ appearing in subformulas of $S$. Suppose now that the object $A=F(q)$ corresponding to a node $q$ in the tree has been defined and its underlying set has been specified. Consider the set of basic covering families over $A$ (which are given by jointly cover sets of arrows of cardinality less than $\kappa$) that witness that some formula in $S$ is forced by $A$ at a given tuple of its underlying domain. That is, if an antecedent or consequent $\eta(\mathbf{x})$ is a (nonempty) disjunction of the form $\bigvee_{i<\gamma}\exists \mathbf{x}_{0}...\mathbf{x}_{n_i} ... \psi_i(\mathbf{x}_0, ...,\mathbf{x}_{n_i}, ..., \mathbf{x})$, and $A \Vdash \eta(\boldsymbol{\beta})$, we include in the set of coverings one of the form $l_j: C_j \to A$, where for each $j$ we have $C_j \Vdash \psi_{i_j}(\boldsymbol{\beta_0^j}, ..., \boldsymbol{\beta_{n_{i_j}}^j}, ..., \boldsymbol{\beta} l_j)$ for some $i_j$ and some $\boldsymbol{\beta_0^j}, ..., \boldsymbol{\beta_{n_{i_j}}^j}, ...$. In case $\eta$ is $\bot$, or $\eta$ is a conjunctive subformula, or $A \nVdash \eta(\boldsymbol{\beta})$ we just consider the identity arrow as a cover. By considering identity arrows if needed, we can also assume that the set of covering families just specified has cardinality $\kappa$.

To construct the functor $F$ by recursion, start with a well-ordering $f: \kappa \times \kappa \to \kappa$ as in Lemma \ref{dwo}, i.e.,  with the property that $f(\beta, \gamma) \geq \gamma$. We describe by an inductive definition how the tree obtained as the image of the functor $F$ is constructed.
 
Suppose therefore that the tree is defined for all levels $\lambda<\mu$ and the covering families over objects corresponding to nodes already defined have cardinality $\kappa$; we show how to define the nodes of level $\mu$. Assume first that $\mu$ is a successor ordinal $\mu=\alpha+1$, and let $\alpha=f(\beta, \gamma)$. Since by hypothesis $f(\beta, \gamma) \geq \gamma$, the nodes $\{p_i\}_{i<m_{\gamma}}$ at level $\gamma$ are defined. Consider the morphisms $g_{ij}^{\alpha}$ over $p_i$ assigned to the paths from each of the nodes $p_i$ to the nodes of level $\alpha$. To define the nodes at level $\alpha+1$, take then the $\beta-th$ covering family over each $p_i$ and pull it back along the morphisms $g_{ij}^{\alpha}$. This produces covering families over each node at level $\alpha$, whose domains are then the nodes of level $\alpha+1$. The underlying domains of such nodes are then formed by considering the elements coming from their predecessors and we add as well the elements coming from witnesses from the $\beta-th$ covering family over the corresponding $p_i$. Suppose now that $\mu$ is a limit ordinal. Then each branch of the tree of height $\mu$ already defined determines a diagram, whose limit is defined to be the node at level $\mu$ corresponding to that branch; its underlying domain is formed by considering all elements coming from predecessor nodes. It is a consequence of the recursion that the underlying domain of each node has cardinality at most $\kappa$, and therefore, since $\kappa^{<\kappa}=\kappa$, the set of covering families over any given object $A$ defined at the inductive step has cardinality at most $\kappa$. By adding identity covers to each set we can assume without loss of generality that it is $\kappa$.
 
 The tree obtained as the image of $F$ has height $\kappa$, and clearly, the morphisms assigned to the paths from any node $p$ till the nodes of level $\alpha$ in the subtree over $p$ form a basic covering family of $p$ because of the transfinite transitivity property. Define now a partial Beth model $B$ over this tree as follows. There is an interpretation of the function symbols in the subset underlying each node which corresponds to composition with the interpretation in the category of the corresponding function symbol. For relations $R$ (including equality), we set by definition $R_q(\mathbf{s}(\boldsymbol{\alpha}))$ if and only if $q$ forces $R(\mathbf{s}(\boldsymbol{\alpha}))$ in the sheaf semantics of the topos, that is, if $q \Vdash R(\mathbf{s}(\boldsymbol{\alpha}))$ (we identify the category with its image through Yoneda embedding). We have now:\\
 
$Claim:$ For every node $p$, every tuple $\boldsymbol{\alpha}$ and every formula $\phi \in S$, $p \Vdash \phi(\boldsymbol{\alpha})$ if and only if $p \Vdash_B \phi(\boldsymbol{\alpha})$, , where $\Vdash_B$ is the forcing in the Beth model.\\

The proof goes by induction on $\phi$.
 
 \begin{enumerate}
  \item If $\phi$ is atomic, the result is immediate by definition of the underlying structures on each node.
  
  \item If $\phi=\bigwedge_{i<\gamma}\psi_i$, the result follows easily from the inductive hypothesis, since we have $p \Vdash \bigwedge_{i<\gamma}\psi_i(\boldsymbol{\alpha})$ if and only if $p \Vdash \psi_i(\boldsymbol{\alpha})$ for each $i<\gamma$, if and only if $p \Vdash_B \psi_i(\boldsymbol{\alpha})$ for each $i<\gamma$, if and only if $p \Vdash_B \bigwedge_{i<\gamma}\psi_i(\boldsymbol{\alpha})$.
  
  \item Suppose $\phi= \bigvee_{i<\gamma}\exists \mathbf{x}_{0}...\mathbf{x}_{n_i} ... \psi_i(\mathbf{x}_0, ...,\mathbf{x}_{n_i}, ..., \mathbf{x})$. If $p \Vdash \phi(\boldsymbol{\beta})$, then there is a basic covering family $\{f_i: A_i \to p\}_{i<\lambda}$ that appears at some point in the well-ordering, such that for each $i<\lambda$, $A_i \Vdash \psi_{i_j}(\boldsymbol{\beta_0^j}, ..., \boldsymbol{\beta_{n_{i_j}}^j}, ..., \boldsymbol{\beta} l_j)$ for some $i_j<\gamma$ and some $\boldsymbol{\beta_0^j}, ..., \boldsymbol{\beta_{n_{i_j}}^j}, ...: A_i : \to [\mathbf{x}, \top]$. Now this covering family is pulled back along all paths $g_j$ of a subtree to create the nodes of a certain level of the subtree over $p$. Hence, every node $m_j$ in such a level satisfies $m_j \Vdash \psi_{i_j}(\boldsymbol{\beta_0^{'j}}, ..., \boldsymbol{\beta_{n_{i_j}}^{'j}}, ..., \boldsymbol{\beta} l_j'))$ for some for some $i_j$ and some $\boldsymbol{\beta_0^{'j}}, ..., \boldsymbol{\beta_{n_{i_j}}^{'j}}, ...$. By inductive hypothesis, $m_j \Vdash_B \psi_{i_j}(\boldsymbol{\beta_0^{'j}}, ..., \boldsymbol{\beta_{n_{i_j}}^{'j}}, ..., \boldsymbol{\beta} l_j'))$, and hence we have $p \Vdash_B \phi(\boldsymbol{\beta})$.
  
  Conversely, if $p \Vdash_B \phi(\boldsymbol{\beta})$, there is a bar over the subtree over $p$ such that for every minimal node $m_j$ there one has $m_j \Vdash_B \psi_{i_j}(\boldsymbol{\beta_0^j}, ..., \boldsymbol{\beta_{n_{i_j}}^j}, ..., \boldsymbol{\beta} f_j)$ for some $i_j<\gamma$ and some $\boldsymbol{\beta_0^j}, ..., \boldsymbol{\beta_{n_{i_j}}^j}, ...: m_j : \to [\mathbf{x}, \top]$, so by inductive hypothesis $m_j \Vdash \psi_{i_j}(\boldsymbol{\beta_0^j}, ..., \boldsymbol{\beta_{n_{i_j}}^j}, ..., \boldsymbol{\beta} f_j)$. Since $\{f_j: m_j \to p\}$ is, by construction, a basic covering family, we must have $p \Vdash \phi(\boldsymbol{\beta})$.
   
  \end{enumerate}
 
\end{proof}

\begin{thm}\label{kgcomp}
If $\kappa$ is a regular cardinal such that $\kappa^{<\kappa}=\kappa$, $\kappa$-geometric theories of cardinality at most $\kappa$ are complete with respect to \Sets-valued models.
\end{thm}

\begin{proof}
 It is enough to prove that every object in the sheaf model forcing the antecedent $\phi(\boldsymbol{\alpha})$ of a valid sequent $\phi \vdash_{\mathbf{x}} \psi$ also forces the consequent $\psi(\boldsymbol{\alpha})$ for every tuple $\boldsymbol{\alpha}$ in the domain. Construct a Beth model over a tree as above but taking as the root of the tree a given object forcing $\phi(\boldsymbol{\alpha})$ and including in the set of formulas $S$ also the $\phi$ and $\psi$; as the underlying domain we include the elements $A \to 1 \to [x, \top]$ coming from the set of constants of $S$ and the tuple $\boldsymbol{\alpha}$. For each branch $\mathbf{b}$ of the tree, consider the directed colimit $\mathbf{D_b}$ of all the underlying structures in the nodes of the branch, with the corresponding functions between them. Such a directed colimit is a structure under the definitions:
 
 \begin{enumerate}
 \item for each function symbol $f$, we define $f(\overline{x_0}, ..., \overline{x_{\lambda}}, ...)=\overline{f(x_0, ..., x_{\lambda}, ...)}$ for some representatives $x_i$ of $\overline{x_i}$; in particular, constants $\mathbf{c}$ are interpreted as $\overline{\mathbf{c}}=\overline{c_0}, ..., \overline{c_{\lambda}}, ...$;
 \item for each relation symbol $R$ we define $R(\overline{x_0}, ..., \overline{x_{\lambda}}, ...) \iff R(x_0, ..., x_{\lambda}, ...)$ for some representatives $x_i$ of $\overline{x_i}$.
 \end{enumerate} 
 
  It is easy to check, using the regularity of $\kappa$, that the structure is well defined and that the choice of representatives is irrelevant. We will show that such a structure is a (possible exploding) positive\footnote{By positive we mean that $\bot$ is not necessarily interpreted as the initial subobject, and by exploding, that it is interpreted as the terminal subobject.} $\kappa$-geometric model of the theory satisfying $\phi(\overline{\boldsymbol{\alpha}})$. Indeed, we have the following:
 
 $Claim:$ Given any $\kappa$-geometric formula $\phi(x_0, ..., x_{\lambda}, ...) \in S$, we have $\mathbf{D_b} \vDash \phi(\overline{\alpha_0}, ..., \overline{\alpha_{\lambda}}, ...)$ if and only if for some node $n$ in the path $\mathbf{b}$, the underlying structure $C_n$ satisfies $C_n \Vdash \phi(\alpha_0, ..., \alpha_{\lambda}, ...)$ for some representatives $\alpha_i$ of $\overline{\alpha_i}$.
 
 The proof of the claim is by induction on the complexity of $\phi$. 
 
 \begin{enumerate}
  \item If $\phi$ is $R(t_0, ..., t_{\lambda}, ...)$ or $s=t$ for given terms $t_i, s, t$, the result follows by definition of the structure. 
  
  \item If $\phi$ is of the form $\bigwedge_{i<\gamma} \theta_i$ the result follows from the inductive hypothesis: $\theta_i$ is forced at some node $n_i$ in the path $\mathbf{b}$, and therefore $\bigwedge_{i<\gamma} \theta_i$ will be forced in any upper bound of $\{n_i: i<\gamma\}$ (here we use the regularity of $\kappa$).
  
  \item If $\phi$ is of the form $\bigvee_{i<\gamma} \theta_i$ and $\mathbf{D_b} \vDash \phi(\overline{\alpha_0}, ..., \overline{\alpha_s}, ...)$, then we can assume that $\mathbf{D_b} \vDash \theta_i(\overline{\alpha_0}, ..., \overline{\alpha_s}, ...)$ for some $i<\gamma$, so that by inductive hypothesis we get $C_n \Vdash \phi(\alpha_1, ..., \alpha_s, ...)$ for some node $n$ in $\mathbf{b}$. Conversely, if $C_n \Vdash \phi(\alpha_0, ..., \alpha_s, ...)$ for some node $n$ in $\mathbf{b}$, by definition of the forcing there is a node $m$ above $n$ in $\mathbf{b}$ and a function $f_{nm}: D_n \to D_m$ for which $C_m \Vdash \theta_i(f_{nm}(\alpha_0), ..., f_{nm}(\alpha_s), ...)$ for some $i<\gamma$, so that by inductive hypothesis we get $\mathbf{D_b} \vDash \phi(\overline{\alpha_0}, ..., \overline{\alpha_s}, ...)$.
  
  \item Finally, if $\phi$ is of the form $\exists \mathbf{x} \psi(\mathbf{x}, x_0, ..., x_s, ...)$ and $\mathbf{D_b} \vDash \phi(\overline{\alpha_0}, ..., \overline{\alpha_s}, ...)$, then $\mathbf{D_b} \vDash \psi(\boldsymbol{\overline{\alpha}}, \overline{\alpha_0}, ..., \overline{\alpha_s}, ...)$ for some $\boldsymbol{\overline{\alpha}}$, and then $C_n \Vdash \psi(\boldsymbol{\alpha}, \alpha_0, ..., \alpha_s, ...)$ for some node $n$ by inductive hypothesis. Conversely, if $C_n \Vdash \phi(\alpha_0, ..., \alpha_s, ...)$ for some node $n$ in $\mathbf{b}$, then by definition of the forcing there is a node $m$ above $n$ in $\mathbf{b}$ and a function $f_{nm}: D_n \to D_m$ for which $C_m \Vdash \psi(f_{nm}(\boldsymbol{\alpha}), f_{nm}(\alpha_0), ..., f_{nm}(\alpha_s), ...)$, which implies that $\mathbf{D_b} \vDash \psi(\boldsymbol{\overline{\alpha}}, \overline{\alpha_0}, ..., \overline{\alpha_s}, ...)$ and hence $\mathbf{D_b} \vDash \phi(\overline{\alpha_0}, ..., \overline{\alpha_s}, ...)$. 
 \end{enumerate}

 Since $\psi(\overline{\boldsymbol{\alpha}})$ is satisfied in all $\kappa$-geometric models of the theory satisfying $\phi(\overline{\boldsymbol{\alpha}})$, it is satisfied in all models of the form $\mathbf{D_b}$ (even if the structure $\mathbf{D_b}$ is exploding). Hence, $\psi(\boldsymbol{\alpha})$ is forced at a certain node of every branch of the tree. Because these nodes form a basic covering family by property $T$, $\psi(\boldsymbol{\alpha})$ is therefore forced at the root, as we wanted to prove. 
\end{proof}

\begin{rmk}
Theorem \ref{kgcomp} is best possible in terms of the cardinality of the theories. Indeed, given an $\kappa^+$-Aronszajn tree (which exists if $\kappa^{<\kappa}=\kappa$, according to \cite{specker}), the theory of a cofinal branch there is obviously geometric and of cardinality $\kappa^+$, but although consistent, it has no model.
\end{rmk}

\section{The $\kappa$-classifying topos of a $\kappa$-geometric theory}

We are now ready to prove the following:

\begin{thm}\label{cltop}
If $\kappa$ is a regular cardinal, any $\kappa$-geometric theory \theory\ has a $\kappa$-classifying topos $\mathcal{B}(\theory)$, defined as a $\kappa$-geometric Grothendieck topos such that there is an equivalence between models of the theory in any other $\kappa$-geometric Grothendieck topos $\mathcal{E}$ and geometric morphisms $\mathcal{E} \to \mathcal{B}(\theory)$ whose inverse images preserve all $\kappa$-small limits.
\end{thm}

\begin{proof}
 We shall show that the topos of sheaves on the syntactic category $\mathcal{C}_{\theory}$ with the $\kappa$-geometric coverage $\tau$ is the $\kappa$-classifying topos of \theory\ . Note that such a topos is $\kappa$-geometric by Lemma \ref{shemb}. We know (see e.g. \cite{johnstone}) that models of the theory in $\mathcal{E}$, i.e., $\kappa$-geometric functors $F: \mathcal{C}_{\theory} \to \mathcal{E}$ are in particular geometric functors and hence they induce a corresponding geometric morphism with direct image $F^*: \mathcal{E} \to \mathcal{B}(\theory)$. We just need to show that, using that $F$ preserves all $\kappa$-small limits and $\mathcal{E}$ has property $T$, the inverse image corresponding to $F^*$ also preserves $\kappa$-small limits. Now if we consider a small subcategory $\mathcal{D}$ of $\mathcal{E}$ closed under $\kappa$-limits, containing a set of generators and the image of $F$, then the corestriction $F: \mathcal{C}_{\theory} \to \mathcal{D}$ can be made into a morphism of sites that preserve $\kappa$-small limits by equipping $\mathcal{D}$ with the topology $\rho$ induced by $\mathcal{E}$, in such a way that $\mathcal{E}=\mathcal{S}h(\mathcal{D}, \rho)$. Therefore, the inverse image corresponding to $F^*$ (its left adjoint), is given by the following composition:
 
$$\mathcal{S}h(\mathcal{C}_{\theory}, \tau) \stackrel{i}{\to} \mathcal{S}et^{\mathcal{C}_{\theory}^{op}} \stackrel{\lim_{F}}{\to} \mathcal{S}et^{\mathcal{D}^{op}} \stackrel{a}{\to} \mathcal{S}h(\mathcal{D}, \tau)$$
\\
where $i$ is the inclusion, $\lim_{F}$ is the left Kan extension of $F$ and $a$ is the associated sheaf functor. Now $i$ being a right adjoint, it preserves all limits; while $\lim_{F}$ preserves all $\kappa$-small limits because these commute with $\kappa$-filtered colimits; more precisely, each $\lim_{F}(-)(D): \mathcal{S}et^{\mathcal{C}_{\theory}^{op}} \to \mathcal{S}et$ preserves $\kappa$-small limits because it is the composition:

$$\mathcal{S}et^{\mathcal{C}_{\theory}^{op}} \stackrel{U^*}{\to} \mathcal{S}et^{(D \downarrow F)^{op}} \stackrel{\lim}{\to} \mathcal{S}et$$
\\
where $U: (D \downarrow F) \to \mathcal{C}_{\theory}$ is the forgetful functor from the comma category. Then $U^*$ preserves all limits, since it has a left adjoint, and $\lim$ preserves $\kappa$-small limits because $(D \downarrow F)^{op}$ is $\kappa$-filtered (which is a consequence of $\mathcal{C}_{\theory}$ having and $F$ preserving $\kappa$-small limits).

Finally, to see that $a$ preserves $\kappa$-small limits note that it is defined with two steps of the plus construction, which is in turn a colimit of sets of matching families over covering sieves ordered by reverse inclusion. Now the topology $\rho$ is generated by a basis consisting of jointly epic families of arrows, and the fact that $\mathcal{E}$ has the exactness property $T$ implies that the transfinite composites (up to $\kappa$) of jointly epic families is jointly epic. In particular, for any set $S$ of $\gamma<\kappa$ basic covering families $(f_{ji_j}: E_{ji_j} \to E)_{i_j<\delta, j<\gamma}$, the family $(h_{g}: P_g \to E_{j g(j)} \to E)_{g \in \delta^{\gamma}}$ (where each $P_g$ is the generalized pullback of the arrows $\{E_{j g(j)} \to E\}_{j < \gamma}$) is jointly covering, being the transfinite composite (up to $\kappa$) of basic covering families, and it factors through every family in $S$. It follows that the set of covering sieves ordered by reverse inclusion is $\kappa$-filtered, and the plus construction preserves, therefore, $\kappa$-small limits. This finishes the proof.
 
\end{proof}

We deduce now the following:

\begin{cor}\label{kp}
 If $\kappa^{<\kappa}=\kappa$, every $\kappa$-separable topos has enough $\kappa$-points, that is, points whose inverse images preserve $\kappa$-small limits.
\end{cor}

\begin{proof}
 First note that every $\kappa$-separable topos is the $\kappa$-classifying topos of a $\kappa$-geometric theory with at most $\kappa$ many axioms. Indeed, it is enough to take the $\kappa$-geometric theory of continuous functors from the underlying category $\mathcal{C}$ of the site that preserve $\kappa$-small limits. Models of this theory in a $\kappa$-geometric topos are precisely continuous functors from $\mathcal{C}$ preserving $\kappa$-small limits, which correspond, by the proof of Theorem \ref{cltop}, to geometric morphisms from the $\kappa$-separable topos whose inverse image preserve $\kappa$-small limits (which is precisely the universal property of the $\kappa$-classifying topos). Because of the $\kappa$-separability, such a theory has at most $\kappa$ many axioms, and therefore it is complete with respect to $\mathcal{S}et$-valued models, by Theorem \ref{kgcomp}. Now any jointly conservative set of $\mathcal{S}et$-valued models of the theory correspond to $\kappa$-points of the $\kappa$-separable topos. To see that these $\kappa$-points are jointly conservative, note that the class of objects $A$ of the $\kappa$-classifying topos such that the inverse images of the $\kappa$-points jointly preserve properness of subobjects contains the objects coming from the syntactic category of the theory and it is closed under coproducts and quotients, so it contains all the objects of the $\kappa$-classifying topos.
\end{proof}

\begin{rmk}
This version of Deligne's theorem for $\kappa$-separable topos can be considered, via Morleyization, as a completeness theorem for $\mathcal{L}_{\kappa^+, \kappa}(T)$, the classical system where we add the axiom scheme given by the rule $T$. The details are in \cite{espindola}.
\end{rmk}

\subsection{Alternative construction of $\mathcal{B}(\theory)$}

We will now prove that if $\kappa^{<\kappa}=\kappa$ (which is a consequence of the Generalized Continuum Hypothesis for every regular $\kappa$), the $\kappa$-classifying topos of a $\kappa$-geometric theory axiomatized by at most $\kappa$-many axioms in canonical form is a $\kappa$-separable topos. We will do this by constructing an alternative site through an infinitary generalization of Coste's version of the classifying topos explained e.g. in \cite{mr} for languages without relation symbols. First, note that every $\kappa$-geometric formula can be put in the canonical form $\bigvee_{i<\delta} \exists \mathbf{x_i}\bigwedge_{j<\gamma}\phi_{ij}$, where $\gamma<\kappa$ and each $\mathbf{x_i}$ has less than $\kappa$ many variables, and where the $\phi_{ij}$ are atomic formulas. This in turn is possible through the use of the axiom of choice and the distributivity axiom $\bigwedge_{i<\gamma} \bigvee_{j<\delta} \phi_{ij} \vdash_{\mathbf{x}}  \bigvee_{f \in \delta^{\gamma}} \bigwedge_{i \in \gamma} \phi_{if(i)}$, both derivable from the rule $T$ (similar derivations are available in \cite{espindola}). A $\kappa$-geometric sequent is in canonical form if it has the form $\bigwedge_{k<\alpha} \psi_k \vdash_{\mathbf{x}} \bigvee_{i<\delta} \exists \mathbf{x_i}\bigwedge_{j<\gamma}\phi_{ij}$ with $\psi_k, \phi_{ij}$ atomic. Every $\kappa$-geometric theory is equivalent to a theory axiomatized in canonical form. 

Note also that by the second of the two methods explained in \cite{johnstone}, D 1.4.9, it is possible to find, for every $\kappa$-geometric theory \theory\ over a signature $\Sigma$ with at most $\kappa$ many axioms in canonical form, a Morita-equivalent\footnote{Two $\kappa$-geometric theories are Morita-equivalent if their categories of models in every $\kappa$-geometric Grothendieck topos are equivalent} theory $\mathbb{T}'$ in a signature $\Sigma'$ with only function symbols and having also $\kappa$ many axioms. For this latter type of theories one can build their $\kappa$-classifying topos as follows. 

The underlying category of the site $\mathcal{C}$ has as objects sets $\Phi(x_0, ..., x_{\alpha}, ...)$ of less than $\kappa$-many atomic formulas (equalities between terms), while a morphism from $\Phi(x_0, ..., x_{\alpha}, ...)$ to $\Psi(x_0, ..., x_{\beta}, ...)$ between two such sets consists of an equivalence class of $\gamma$-tuples of terms $(t_0, ..., t_{\beta}, ...)$ of the same type as the free variables in $\Psi$ and with free variables among those of $\Phi$, such that the sequent $\bigwedge \Phi \vdash_{\{x_0, ..., x_{\alpha}, ...\}} \bigwedge \Psi(t_0, ..., t_{\beta}, ...)$ is provable, and where two tuples $(t_i)_{i<\gamma}$ and $(s_i)_{i<\gamma}$ are equivalent if the sequent $\bigwedge \Phi \vdash_{\{x_0, ..., x_{\alpha}, ...\}} \bigwedge_{i<\gamma}t_i=s_i$ is provable. Composition is given by substitution. This category is the dual of the full subcategory $\mathcal{P}$ of $\kappa$-presentable algebras on the signature of the theory (see \cite{mr}), through the assignment that sends an object $\Phi$ of $\mathcal{C}$ to the algebra in the generators given by the free variables of $\Phi$ and the equations in $\Phi$. In particular, $\mathcal{C}$ has all $\kappa$-small limits.

There is a Grothendieck topology associated to the axiomatization. More precisely, to each axiom $\bigwedge \Phi \vdash_{\mathbf{x}} \bigvee_{i<\delta} \exists \mathbf{x_i}\bigwedge\Psi_{i}$ we define a covering family to be the obvious set of morphisms $\{\Phi \cup \Psi_i \to \Phi\}_{i<\delta}$. Then we define the $\kappa$-Grothendieck topology $\rho$ generated by these set of covers.

There is a $\Sigma'$-structure in $\mathcal{C}=\mathcal{P}^{op}$ defined as follows. To each sort $S$ we assign the free algebra $F[x]$ on one generator $x: S$, while to functions $f: S_0 \times ...S_{\alpha} \times ... \to S$ we assign the unique morphism $t: F[x] \to F[x_0, ..., x_{\alpha}, ...]$ such that $t(x)=f(x_0, ..., x_{\alpha}, ...)$. In a similar way as explained in \cite{mr}, we can prove that the topos of sheaves on $\mathcal{C}$ satisfies the universal property of the $\kappa$-classifying topos of the theory with respect to models in $\mathcal{S}et$. We claim now the following:

\begin{thm}
$\mathcal{S}h(\mathcal{C}, \rho)$ is the $\kappa$-classifying topos of \theory\ . 
\end{thm}

\begin{proof}
 We prove that $\mathcal{S}h(\mathcal{C}, \rho)$ is equivalent to $\mathcal{S}h(\mathcal{C}_{\mathbb{T}'}, \tau)$ ($\tau$ being the $\kappa$-geometric coverage) by showing that $\mathcal{S}h(\mathcal{C}, \rho)$ has the universal property of the $\kappa$-classifying topos with respect to a $\kappa$-geometric topos with enough $\kappa$-points. Given such a topos $\mathcal{E}$, it is easy to prove that there is a conservative $\kappa$-geometric morphism with inverse image $E: \mathcal{E} \to \mathcal{S}et^{I}$ such that composition with the evaluation at $i \in I$, $ev(i)E$ gives a $\kappa$-point of $\mathcal{E}$. Now each model of $\mathbb{T}'$ in $\mathcal{E}$ give rise to models in $\mathcal{S}et$ by considering their images through each $ev(i)E$. These correspond to unique (up to isomorphism) $\kappa$-geometric morphisms with inverse image $\mathcal{S}h(\mathcal{C}, \rho) \to \mathcal{S}et$, which in turn induce a $\kappa$-geometric morphism with inverse image $G: \mathcal{S}h(\mathcal{C}, \rho) \to \mathcal{S}et^I$ and with the property that the composition $Gay: \mathcal{C} \to \mathcal{S}et^{\mathcal{C}^{op}} \to \mathcal{S}h(\mathcal{C}, \rho) \to \mathcal{S}et^I$ maps the product of sorts in the $\Sigma'$-structure in $\mathcal{C}$ into $\mathcal{E}$. Now $G$ preserves $\kappa$-small limits and colimits, and every object in $\mathcal{S}h(\mathcal{C}, \rho)$ is a colimit of objects of the form $ay(C)$, while every object $C$ in $\mathcal{C}=\mathcal{P}^{op}$, as a $\kappa$-presentable algebra, is a $\kappa$-small limit of objects corresponding to the free $\kappa$-presentable algebras $F[x]$ in $\mathcal{P}^{op}$, that correspond in turn to sorts in the $\Sigma'$-structure in $\mathcal{C}$. Therefore, since $ay$ also preserves $\kappa$-small limits, $G$ is completely determined (up to isomorphism) by its value on the objects $ay(C)$ for $C$ a sort in the in the $\Sigma'$-structure in $\mathcal{C}$. Since the value of $G$ on such objects belongs to $\mathcal{E}$, and $E$ preserves $\kappa$-small limits and colimits, it follows that $G$ itself factors through $\mathcal{E}$. Moreover, it is the unique (up to isomorphism) inverse image of a $\kappa$-geometric morphism corresponding to the given model in $\mathcal{E}$. This finishes the proof.
\end{proof}

We now immediately get:

\begin{cor}\label{sep}
 If $\kappa^{<\kappa}=\kappa$, then the $\kappa$-classifying topos of a $\kappa$-geometric theory of at most $\kappa$-many axioms in canonical form is $\kappa$-separable.
\end{cor}

\subsection{$\kappa$-coherent toposes}

Let us now assume that $\kappa$ is a weakly compact cardinal. $\kappa$-coherent logic is the fragment of $\kappa$-geometric logic where disjunctions are indexed by ordinals less than $\kappa$. In this case, because every bar over the tree $\gamma^\kappa$ (for $\gamma<\kappa$) is uniform\footnote{A bar is uniform if it contains all nodes in a given level of the tree.} (a consequence of the weak compactness of $\kappa$), it is possible to replace rule $T$ with the transfinite transitivity property of \cite{espindola}. A $\kappa$-coherent topos is a topos that occurs as the $\kappa$-classifying topos of a $\kappa$-coherent theory of cardinality at most $\kappa$. Alternatively, it is a $\kappa$-geometric topos on a site of size at most $\kappa$ whose topology is generated by families of less than $\kappa$ many morphisms. By Corollary \ref{kp}, the $\kappa$-classifying topos of a $\kappa$-coherent theory of cardinality at most $\kappa$ has enough $\kappa$-points. Therefore we get:

\begin{thm}
 If $\kappa$ is weakly compact, every $\kappa$-coherent topos has enough $\kappa$-points.
\end{thm}

In the same way Deligne's theorem can be considered as G\"odel's completeness theorem for $\mathcal{L}_{\omega, \omega}$, so this infinitary version can be considered as Karp's completeness for $\mathcal{L}_{\kappa, \kappa}$. To see this, note that we have:

\begin{cor}\label{cohcomp}
 $\kappa$-coherent theories of cardinality at most $\kappa$ are complete with respect to $\mathcal{S}et$-valued models.
\end{cor}

\begin{proof}
 Construct the $\kappa$-coherent syntactic category $\mathcal{C}_{\theory}$ of the $\kappa$-coherent theory and equip it with the Grothendieck topology $\tau$ whose basis consists of jointly epic families of less than $\kappa$ many morphisms. By the proof of Theorem \ref{cltop}, it follows that $\mathcal{S}h(\mathcal{C}_{\theory}, \tau)$ is the $\kappa$-classifying topos of the theory. By composing its $\kappa$-points with the embedding $y: \mathcal{C}_{\theory} \to \mathcal{S}h(\mathcal{C}_{\theory}, \tau)$ we get a jointly conservative family of models.
\end{proof}

As shown in \cite{espindola}, the transfinite transitivity rule is, in the Boolean case, equivalent to the addition of the axioms of distributivity and dependent choice from \cite{karp}, and so Theorem \ref{cohcomp} is essentially Karp's completeness theorem for $\mathcal{L}_{\kappa, \kappa}$.

As a final remark, we mention that in case $\kappa$ is strongly compact, the restriction on the cardinality of the $\kappa$-coherent theory can be removed. In this case, a $\kappa$-coherent topos is a $\kappa$-geometric topos on a site of arbitrary size whose topology is generated by families of less than $\kappa$ morphisms. Using the completeness of $\kappa$-coherent theories of arbitrary size (see \cite{espindola}) for this case, it follows that any topos that is $\kappa$-coherent in this sense has enough $\kappa$-points.

\bibliographystyle{amsalpha}

\renewcommand{\bibname}{References} 

\bibliography{references}



\end{document}